\documentclass[a4j,14pt]{article}  
\usepackage{geometry}
\usepackage[T1]{fontenc} 
\usepackage{times} 
\usepackage{amsmath,amssymb,bm, mathpazo}
\usepackage{amsfonts}
\usepackage{pifont}
\usepackage{mathrsfs}
\usepackage{graphicx}
\usepackage{amscd}
   
\makeatletter

\newtheorem{thm}{\theoremname}[section]
\newtheorem{prop}[thm]{\propositionname}
\newtheorem{cor}[thm]{\corollaryname}
\newtheorem{df}[thm]{\definitionname}
\newtheorem{lemma}[thm]{\lemmaname}

\newtheorem{ex}[thm]{\examplename}
\newtheorem{remark}[thm]{\remarkname}

\newcommand{\theoremname}{Theorem}
\newcommand{\definitionname}{Definition}
\newcommand{\lemmaname}{Lemma}
\newcommand{\corollaryname}{Corollary}
\newcommand{\axiomname}{Axiom}
\newcommand{\propositionname}{Proposition}
\newcommand{\problemname}{Problem}
\newcommand{\examplename}{Example}
\newcommand{\remarkname}{Remark}

\newcommand{\conjecturename}{Conjecture}

\def\tedsymbol{\vcenter{\hbox{\vrule\@height.5em\@width.5em}}}
\def\ted{{\unskip\nobreak\hfil\penalty50
 \quad\hbox{}\nobreak\hfil \hbox{$\tedsymbol$}
 \parfillskip\z@ \finalhyphendemerits\z@\par}}

\providecommand{\abs}[1]{\lvert#1\rvert}

\makeatother

\title{Higher Tate central extensions via $K$-theory and $\infty$-topos theory}
\author{
\LARGE{Sho Saito}\footnote{Supported by JSPS Research Felowships for Young Scientists.} 
}
\begin{document}
\date{\empty}
\maketitle
\nocite{Harada-01,Katura-05,Nakano-03,Takagi-96,Artin-74,Kuga-68,MP-64,Infeld-96,Borceux-01,Serre-95,Iyanaga-99,Iyanaga-02}
\begin{abstract}
We give a classification theorem of certain geometric objects, called \textit{torsors over the sheaf of $K$-theory spaces}, in terms of Tate vector bundles. 
This allows us to present a very natural and simple, alternative approach to the Tate central extension, which was classically constructed by using the gerbe of determinant theories. 
We use the language of $\infty$-topoi as the theoretical framework, since it has well-developed, extended notions of groups, actions, and torsors, which make it possible to regard the sheaf of $K$-theory spaces as a group object of such kind and to interpret a delooping theorem in $K$-theory as a classification theorem for torsors over it. 
\end{abstract}
\section{Introduction} 
\subsection*{Background and aim}
In the representation theory of loop groups, one often encounters with situations where $\mathbb{G}_m$-central extensions of a loop group are concerned. 
There is a canonical one among such, the \textit{Tate central extension}, which appears as a pullback of a more general construction in the infinite-dimensional linear algebra of Tate vector spaces. 
A topological vector space over a discrete field $k$ is called a \textit{Tate vector space} if it is isomorphic to the direct sum of a discrete space and the dual of a discrete space. 
A typical example of a Tate vector space is the space $k((t))$ of formal Laurent series with the $t$-adic topology. 
If $G$ is a reductive algebraic group and $V$ a finite dimensional representation then there is an induced natural representation of the corresponding loop group $G((t))$ on a Tate vector space $V((t))$. 

The group of automorphisms of a Tate vector space is known to have a canonical $\mathbb{G}_m$-central extension, called the \textit{Tate central extension}, for whose construction we refer the reader to, for example, \cite{kapranov}. 
The Tate central extension is classified by a $\mathbb{G}_m$-gerbe equipped with an action by the automorphism group, but the assignment of this gerbe to each Tate vector space is not canonically compatible with direct sums. 
This led Beilinson et al. \cite{bbe} and Drinfeld \cite{drinfeld} to introduce the notion of a \textit{torsor over a sheaf of Picard groupoids}, enriching the $\mathbb{G}_m$-gerbe classifying the Tate central extension to a $\operatorname{Pic}^{\mathbb{Z}}$-torsor classifying an object that should be called the \textit{categorical Tate central extension} of the automorphism group of a Tate vector bundle by the stack $\operatorname{Pic}^{\mathbb{Z}}$ of $\mathbb{Z}$-graded line bundles. 
See \cite{bbe}, section 2, and \cite{drinfeld}, secion 5, for details. 

They gave the construction of the $\operatorname{Pic}^{\mathbb{Z}}$-torsor by a direct analogy of the classical construction of the plain $\mathbb{G}_m$-gerbe as in \cite{kapranov}, but Drinfeld proposes in section 5.5 of \cite{drinfeld} an interesting idea, which he attributes to Beilinson. 
Their idea, posed as a "somewhat vague picture," roughly says that there should be a more homotopical interpretation of the $\operatorname{Pic}^{\mathbb{Z}}$-torsor classifying the categorical Tate central extension in terms of algebraic $K$-theory. 
Drinfeld's description of their idea remains in a sketchy state (which is why it is called a "vague picture"), and he leaves it as a problem to make it precise. 

The aim of this article is to propose and prove a more precisely and more comprehensively formulated version of Beilinson-Drinfeld's picture, presenting a very natural and simple approach to the Tate central extension via a classification theorem of objects called \textit{torsors over the sheaf of $K$-theory spaces}. 
The theory of $\infty$-topoi, recently developed by Lurie \cite{htt} et al., makes it possible to regard the whole sheaf of $K$-theory spaces (note that the stack of graded line bundles $\operatorname{Pic}^{\mathbb{Z}}$ can be interpreted as a truncation of the $K$-theory sheaf) as a group object, allowing us to meaningfully speak of torsors over it.  
We show that the corresponding classifying space is equivalent to the $K$-theory sheaf of Tate vector bundles, as a geometric consequence of a delooping theorem obtained by the author in \cite{saito} and Drinfeld's theorem that the first negative $K$-group vanishes Nisnevich locally (\cite{drinfeld}, Theorem 3.4). 
This directly leads to a canonical construction of a torsor over the sheaf of $K$-theory spaces to each Tate vector bundle. 
The torsor thus obtained admits a canonical action by the sheaf of automorphisms of the Tate vector bundle, thereby resulting an object that should be called the \textit{$\infty$-categorical Tate central extension} of the automorphism group of the Tate vector bundle by the sheaf of $K$-theory spaces.  

We believe that our approach via a delooping theorem of $K$-theory, or its geometric consequence in an $\infty$-topos where the $K$-theory satisfies descent and the delooped $K$-theory satisfies local connectedness, is the most comprehensive and conceptually appropriate way of treating the Tate central extension. 
We will discuss a possible generalization of the results presented here to more higher dimensional contexts in future work. 

\subsection*{Summary of the results}
Let us give here a more detailed and precise summary of our results. 

Write $\Pi$ for the filtered category of pairs $(i,j)$ of integers with $i\leq j$, where there is a unique morphism $(i,j)\to(i^{\prime},j^{\prime})$ if $i\leq i^{\prime}$ and $j\leq j^{\prime}$. 
For an exact category $\mathcal{A}$, let $\displaystyle\lim_{\longleftrightarrow}\mathcal{A}$ be the full subcategory of $\operatorname{Ind}\operatorname{Pro}\mathcal{A}$ consisting of ind-pro-objects $X=(X_{i,j})_{(i,j)\in\Pi}$, indexed by $\Pi$, satisfying that for every $i\leq j\leq k$ the sequence $$0\to X_{i,j}\to X_{i,k}\to X_{j,k}\to0$$ is a short exact sequence in $\mathcal{A}$. 
If the exact category $\mathcal{A}$ is an extension-closed, full additive subcategory of an abelian category $\mathcal{F}$, then $\displaystyle\lim_{\longleftrightarrow}\mathcal{A}$ is an extension-closed, full additive subcategory of the abelian category $\operatorname{Ind}\operatorname{Pro}\mathcal{F}$, so that $\displaystyle\lim_{\longleftrightarrow}\mathcal{A}$ is endowed with a structure of an exact category. 
See \cite{beilinson}, A.3, and \cite{lcoec}, for details on the exact category $\displaystyle\lim_{\longleftrightarrow}\mathcal{A}$. 

We write $\mathbb{K}$ for Schlichting's non-connective $K$-theory spectrum of an exact category, introduced in \cite{schlichting2}, whose positive homotopy groups are the positive $K$-groups of the exact category, and whose $0$-th homotopy group is the $0$-th $K$-group of the idempotent completion of the exact category, and whose negative homotopy groups recover the classical negative $K$-groups when the exact category is the category of finitely generated projective modules over a ring or the category of vector bundles on a quasi-compact, quasi-separated scheme with an ample family of line bundles. 
See \cite{schlichting2} for details. 
In a recent paper \cite{saito} the author proved the following theorem. 
(Note that idempotent completion causes no change on non-connective $K$-theory.) 
\begin{thm}[\cite{saito}, 1.2]
\label{abstractdelooping}
There is a natural equivalence of sectra between $\mathbb{K}(\mathcal{A})$ and $\Omega\mathbb{K}((\displaystyle\lim_{\longleftrightarrow}\mathcal{A})^{\natural})$, where $(-)^{\natural}$ denotes the idempotent completion.  
\end{thm}
\begin{remark}
\begin{enumerate}
\item This was the concluding conjecture of L. Previdi's thesis (\cite{previdi}, Conjecture 5.1.7). 
\item In the case where $\mathcal{A}$ is the category of finitely generated projective $R$-modules, Drinfeld \cite{drinfeld} observes a fact which is essentially the $\pi_{-1}$-part of this equivalence: That is, he observes an isomorphism between the first negative $K$-group of $R$ and the $0$-th $K$-group of his category of Tate $R$-modules (\cite{drinfeld}, Theorem 3.6-(iii)). 
\item Recent work of Br\"{a}unling, Grochenig and Wolfson \cite{bgw2} provides an interpretation of this theorem as an algebraic analogue of the Atiyah-Janich theorem in topological $K$-theory.  
\end{enumerate}
\end{remark}

Let $R$ be a commutative ring, which we assume in the sequel to be noetherian and of finite Krull dimension, and denote by $\mathcal{P}(R)$ the exact category of finitely generated projective $R$-modules.  
Then the idempotent-completed exact category $(\displaystyle\lim_{\longleftrightarrow}\mathcal{P}(R))^{\natural}$ is very close to Drinfeld's category $\operatorname{Tate}_R^{\operatorname{Dr}}$ of Tate $R$-modules (which is denoted by $\mathcal{T}_R$ in \cite{drinfeld}, 3.3.2). 
Indeed, if $(M_{i,j})_{i\leq j}$ is an object of $\displaystyle\lim_{\longleftrightarrow}\mathcal{P}(R)$, the $R$-module $\varinjlim_j\varprojlim_iM_{i,j}$ endowed with the topology induced from the discrete ones on $M_{i,j}$ is an elementary Tate $R$-module in Drinfeld's sense (\cite{drinfeld}, 3.2.1). 
Recent work by Br\"{a}unling, Grochenig, and Wolfson \cite{bgw} shows this induces a fully faithful functor $(\displaystyle\lim_{\longleftrightarrow}\mathcal{P}(R))^{\natural}\hookrightarrow\operatorname{Tate}_R^{\operatorname{Dr}}$, which is an equivalence onto the full subcategory of Tate $R$-modules of countable type (that is, direct summands of elementary Tate $R$-modules $P\oplus Q^{\ast}$ where $P$ and $Q$ are countably generated discrete, projective $R$-modules). 
See \cite{bgw}, Theorem 5.22. 

\begin{df} 
We call $(\displaystyle\lim_{\longleftrightarrow}\mathcal{P}(R))^{\natural}$ the category of {\rm Tate vector bundles} over the affine scheme $\operatorname{Spec}R$. 
\end{df}

We write $\operatorname{Spec}R_{\operatorname{Nis}}$ for the site whose underlying category is the opposite category of \'{e}tale $R$-algebras and $R$-homomorphisms, and whose notion of a covering is given as follows. 
A collection of \'{e}tale morphisms $\{\operatorname{Spec}R^{\prime}_{\alpha}\to\operatorname{Spec}R^{\prime}\}_{\alpha\in A}$ over $\operatorname{Spec}R$ is a covering in $\operatorname{Spec}R_{\operatorname{Nis}}$ if it is the opposite of a family of \'{e}tale $R$-homomorphisms $\{\phi_{\alpha}:R^{\prime}\to R_{\alpha}^{\prime}\}_{\alpha\in A}$ for which there exists a finite sequence of elements $a_1,\ldots,a_n\in R^{\prime}$ such that $(a_1,\ldots,a_n)=R^{\prime}$ and for every $1\leq i\leq n$ there exists an $\alpha\in A$ and an $R$-homomorphism $\psi:R^{\prime}_{\alpha}\to R^{\prime}[\frac{1}{a_i}]/(a_1,\ldots,a_{i-1})$ whose composition with $\phi_{\alpha}:R^{\prime}\to R^{\prime}_{\alpha}$ equals the map $R^{\prime}\to R^{\prime}[\frac{1}{a_i}]/(a_1,\ldots,a_{i-1})$. 
(See \cite{dagxi}, section 1, for details.) 
We refer to $\operatorname{Spec}R_{\operatorname{Nis}}$ as the \textit{small Nisnevich site} of the affine scheme $\operatorname{Spec}R$. 

Denote by $\operatorname{Set}_{\Delta}$ the the category of simplicial sets, which is a combinatorial, simplicial model category with the Kan model structure.  
We write $\operatorname{Set}_{\Delta}^{\operatorname{Spec}R_{\operatorname{Nis}}^{\operatorname{op}}}$ for the combinatorial, simplicial model category of simplicial presheaves on the underlying category of $\operatorname{Spec}R_{\operatorname{Nis}}$ with the injective model structure, and $(\operatorname{Set}_{\Delta}^{\operatorname{Spec}R_{\operatorname{Nis}}^{\operatorname{op}}})^{\circ}$ for its fibrant-cofibrant objects. 
 By Proposition 4.2.4.4 of \cite{htt}, there is an equivalence of $\infty$-categories $$\theta:N(\operatorname{Set}_{\Delta}^{\operatorname{Spec}R_{\operatorname{Nis}}^{\operatorname{op}}})^{\circ}\stackrel{\sim}{\to}\operatorname{Fun}(N\operatorname{Spec}R_{\operatorname{Nis}}^{\operatorname{op}},(\operatorname{Spaces}))=\operatorname{Preshv}_{(\operatorname{Spaces})}(N\operatorname{Spec}R_{\operatorname{Nis}}),$$ where $N$ denotes the simplicial nerve and $(\operatorname{Spaces})$ is the $\infty$-category of spaces, which is by definition the simplicial nerve of the simplicial category of Kan complexes.    
(We write $\operatorname{Preshv}_{(\operatorname{Spaces})}(\mathcal{C})$ for the $\infty$-category of presheaves of spaces on an $\infty$-category $\mathcal{C}$.) 
Let $\operatorname{Set}_{\Delta,\operatorname{loc}}^{\operatorname{Spec}R_{\operatorname{Nis}}^{\operatorname{op}}}$ denote the combinatorial, simplicial model category of simplicial presheaves on the site $\operatorname{Spec}R_{\operatorname{Nis}}$ with respect to Jardine's local model structure \cite{jardine}, and $(\operatorname{Set}_{\Delta,\operatorname{loc}}^{\operatorname{Spec}R_{\operatorname{Nis}}^{\operatorname{op}}})^{\circ}$ its fibrant-cofibrant objects. 
Then Proposition 6.5.2.14 of \cite{htt} shows that the above equivalence $\theta$ restricts to the equivalence $$\theta:N(\operatorname{Set}_{\Delta,\operatorname{loc}}^{\operatorname{Spec}R_{\operatorname{Nis}}^{\operatorname{op}}})^{\circ}\stackrel{\sim}{\to}\operatorname{Shv}_{(\operatorname{Spaces})}(N\operatorname{Spec}R_{\operatorname{Nis}})^{\wedge}\subset\operatorname{Shv}_{(\operatorname{Spaces})}(N\operatorname{Spec}R_{\operatorname{Nis}}),$$ 
where $\operatorname{Shv}_{(\operatorname{Spaces})}(N\operatorname{Spec}R_{\operatorname{Nis}})\subset\operatorname{Preshv}_{(\operatorname{Spaces})}(N\operatorname{Spec}R_{\operatorname{Nis}})$ is the $\infty$-topos of sheaves of spaces on $N\operatorname{Spec}R_{\operatorname{Nis}}$ (see Definition \ref{inftytopos} and Example \ref{sheavesonsites} below), and $(-)^{\wedge}$ denotes its hypercompletion (\cite{htt}, 6.5.2). 

Suppose $R$ is noetherian and of finite Krull dimension. 
Then, by Thomason's Nisnevich descent theorem of non-connective $K$-theory (\cite{tt}, 10.8), the simplicial presheaf on $\operatorname{Spec}R_{\operatorname{Nis}}$ given by $K$-theory spaces $$R^{\prime}\mapsto\Omega^{\infty}\mathbb{K}(R^{\prime})$$ is a fibrant object of $\operatorname{Set}_{\Delta,\operatorname{loc}}^{\operatorname{Spec}R_{\operatorname{Nis}}^{\operatorname{op}}}$, so that by the above equivalence $\theta$ it defines an object of the $\infty$-topos $\operatorname{Shv}_{(\operatorname{Spaces})}(N\operatorname{Spec}R_{\operatorname{Nis}})$. 
\begin{df}
\label{k}
We denote this object by $$\mathcal{K}=\theta(\Omega^{\infty}\mathbb{K}(-))\in\operatorname{ob}\operatorname{Shv}_{(\operatorname{Spaces})}(N\operatorname{Spec}R_{\operatorname{Nis}}).$$ 
\end{df}
Note that a presheaf of spectra satisfies Nisnevich descent if and only if it sends elementary Nisnevich squares to pullback-pushout squares. 
Since the suspension functor $\Sigma$ preserves pullback-pushout squares of spectra, we see that the Nisnevich descent of the non-connective $K$-theory $\mathbb{K}(-)$ implies the Nisnevich descent of $\Sigma\mathbb{K}(-)$, which is equivalent by Theorem \ref{abstractdelooping} to the presheaf $\mathbb{K}((\displaystyle\lim_{\longleftrightarrow}\mathcal{P}(-))^{\natural})$. 
Hence the simplicial presheaf on $\operatorname{Spec}R_{\operatorname{Nis}}$ given by $$R^{\prime}\mapsto\Omega^{\infty}\mathbb{K}((\displaystyle\lim_{\longleftrightarrow}\mathcal{P}(R^{\prime}))^{\natural})$$ is also fibrant in $\operatorname{Set}_{\Delta,\operatorname{loc}}^{\operatorname{Spec}R_{\operatorname{Nis}}^{\operatorname{op}}}$, and thus defines, via the equivalence $\theta$, an object of the $\infty$-topos $\operatorname{Shv}_{(\operatorname{Spaces})}(N\operatorname{Spec}R_{\operatorname{Nis}})$. 
\begin{df}
\label{ktate}
We denote this object by $$\mathcal{K}_{\operatorname{Tate}}=\theta(\Omega^{\infty}\mathbb{K}((\displaystyle\lim_{\longleftrightarrow}\mathcal{P}(-))^{\natural})\in\operatorname{ob}\operatorname{Shv}_{(\operatorname{Spaces})}(N\operatorname{Spec}R_{\operatorname{Nis}}).$$ 
\end{df}

We refer the reader to section 2 for a short exposition on the materials of $\infty$-topos theory employed in this article, which are collected from \cite{htt} and \cite{nss}. 
We in particular make essential use of the notions of group objects, their actions, and torsors, in an $\infty$-topos. 
These notions we recall in section 2, Definitions \ref{group}, \ref{action}, and \ref{torsor}, respectively, following \cite{htt} and \cite{nss}. 
\begin{prop}
\label{Kasagroup}
The object $\mathcal{K}$ is a group object in the $\infty$-topos $\operatorname{Shv}_{(\operatorname{Spaces})}(N\operatorname{Spec}R_{\operatorname{Nis}})$.  
\end{prop}
\begin{df}
We refer as a {\rm torsor over the sheaf of $K$-theory spaces} to a $\mathcal{K}$-torsor over the final object $\operatorname{Spec}R$, where $\mathcal{K}$ is regarded by Proposition \ref{Kasagroup} as a group object in the $\infty$-topos $\operatorname{Shv}_{(\operatorname{Spaces})}(N\operatorname{Spec}R_{\operatorname{Nis}})$. 
\end{df}

In general, for every group object $G$ of an $\infty$-topos $\mathfrak{X}$ there is an object $BG$ that classifies $G$-torsors, in the sense that for each object $X$ of $\mathfrak{X}$ there is an equivalence between the $\infty$-groupoid of $G$-torsors over $X$ and the mapping space from $X$ to $BG$; the object $BG$ is just given by the connected delooping of the group object $G$. 
(Theorem 3.19 of \cite{nss}, recalled in section 2 below as Theorem \ref{3.19}.) 
We call the object $BG$ the \textit{classifying space object} of the group object $G$. 

The following is the geometric incarnation of Theorem \ref{abstractdelooping}, which serves as a classification theorem of torsors over the sheaf of $K$-theory spaces. 
We remark that Drinfeld's theorem on the Nisnevich local vanishing of the first negative $K$-group (\cite{drinfeld}, Theorem 3.4) also plays a crucial role in its proof. 
\begin{thm}
\label{geometricdelooping}
The classifying space object of the group object $\mathcal{K}$ in the $\infty$-topos $\operatorname{Shv}_{(\operatorname{Spaces})}(\operatorname{Spec}R_{\operatorname{Nis}})$ is given by the object $\mathcal{K}_{\operatorname{Tate}}$. 
I.e., there is an equivalence between $B\mathcal{K}$ and $\mathcal{K}_{\operatorname{Tate}}$. 
\end{thm}
\begin{cor}
\label{dm}
Torsors over the sheaf of $K$-theory spaces is classified by points of the space $\mathcal{K}_{\operatorname{Tate}}(\operatorname{Spec}R)=\Omega^{\infty}\mathbb{K}((\displaystyle\lim_{\longleftrightarrow}\mathcal{P}(R))^{\natural})$. 
In particular, a Tate vector bundle $M\in\operatorname{ob}(\displaystyle\lim_{\longleftrightarrow}\mathcal{P}(R))^{\natural}$ defines a torsor $\mathfrak{D}_M$ over the sheaf of $K$-theory spaces. 
\end{cor}
Let $\operatorname{Aut}M$ denote the sheaf of groups on $\operatorname{Spec}R_{\operatorname{Nis}}$ given by $$R^{\prime}\mapsto\operatorname{Aut}_{(\displaystyle\lim_{\longleftrightarrow}\mathcal{P}(R^{\prime}))^{\natural}}M\otimes_RR^{\prime}.$$
This is a group object of the ordinary topos $\operatorname{Shv}_{(\operatorname{Sets})}(\operatorname{Spec}R_{\operatorname{Nis}})$, which is regarded as the full subcategory of discrete objects of the $\infty$-topos $\operatorname{Shv}_{(\operatorname{Spaces})}(N\operatorname{Spec}R_{\operatorname{Nis}})$. 
\begin{thm}
\label{autmaction}
There is a canonical action of the group object $\operatorname{Aut}M$ on the $\mathcal{K}$-torsor $\mathfrak{D}_M$.  
\end{thm}  

\subsection*{Organization and conventions}
Section 2 provides a brief review of the necessary materials in $\infty$-topos theory, main references being \cite{htt} and \cite{nss}. 
In section 3 we prove our results. 
We work in an $\infty$-categorical setting and refer the reader to \cite{htt} for basic terminology. 
In this article the category of simplicial sets is denoted by $\operatorname{Set}_{\Delta}$. 
We write $(\operatorname{Spaces})$ for the $\infty$-category of spaces (the simplicial nerve of the simplicial category of Kan complexes), and denote by $\operatorname{Preshv}_{(\operatorname{Spaces})}(\mathcal{C})$ and $\operatorname{Shv}_{(\operatorname{Spaces})}(\mathcal{C})$ the $\infty$-categories of presheaves and sheaves of spaces on $\mathcal{C}$, respectively.   

\subsection*{Acknowledgement}
I thank Luigi Previdi for answering my questions on his conjecture. 
In fact, the work presented here grew out from an endeavour to understand the conceptual meaning of Theorem \ref{abstractdelooping}, and his suggestion that the main application of it should be to higher Tate central extensions was truly helpful.  

\section{Recollection on the theory of $\infty$-topoi}
In this section we give a review on the necessary materials of $\infty$-topos theory, collected from \cite{htt} and \cite{nss}. 
Since our exposition given here is somewhat terse, we refer the reader to the references \cite{htt} and \cite{nss} for the full details. 

Let us begin with recalling the definition of an $\infty$-topos. 
\begin{df}[$\infty$-topos; \cite{htt}, 6.1.0.4]
\label{inftytopos}
An {\rm $\infty$-topos} $\mathfrak{X}$ is a full, accessible (see \cite{htt}, 5.4.2.1, for the definition) subcategory of the $\infty$-category $\operatorname{Preshv}_{(\operatorname{Spaces})}(\mathcal{C})=\operatorname{Fun}(\mathcal{C}^{\operatorname{op}},(\operatorname{Spaces}))$ of presheaves of spaces on some $\infty$-category $\mathcal{C}$, such that the inclusion $\mathfrak{X}\hookrightarrow\operatorname{Preshv}_{(\operatorname{Spaces})}(\mathcal{C})$ has a left adjoint which preserves finite limits. 
The left adjoint is called the {\rm sheafification functor}.  
\end{df}
A typical example of an $\infty$-topos is the $\infty$-category $\operatorname{Shv}_{(\operatorname{Spaces})}(\mathcal{C})$ of sheaves of spaces on an $\infty$-category $\mathcal{C}$ equipped with a Grothendieck topology. 
\begin{ex}[$\infty$-topos of sheaves of spaces; \cite{htt}, 6.2.2]
\label{sheavesonsites}
Let $\mathcal{C}$ be an $\infty$-category. 
A {\rm sieve} on an object $C$ of $\mathcal{C}$ is a full subcategory $\mathcal{C}^{(0)}_{/C}$ of the overcategory $\mathcal{C}_{/C}$ such that if a morphism in $\mathcal{C}_{/C}$ has its target in $\mathcal{C}^{(0)}_{/C}$ then it also has its source in $\mathcal{C}^{(0)}_{/C}$.  
It is Proposition 6.2.2.5 of \cite{htt} that there is a canonical bijection between sieves on the object $C$ and monomorphisms in the $\infty$-category $\operatorname{Preshv}_{(\operatorname{Spaces})}(\mathcal{C})$ whose target is $j(C)$, where $j:\mathcal{C}\hookrightarrow\operatorname{Preshv}_{(\operatorname{Spaces})}(\mathcal{C})$ denotes the Yoneda embedding.  

A {\rm Grothendieck topology} on $\mathcal{C}$ is an assignment of a collection of sieves on $C$ to each object $C$ of $\mathcal{C}$. 
A sieve on $C$ belonging to that assigned collection is called a {\rm covering sieve} on $C$. 
A presheaf $F\in\operatorname{Preshv}_{(\operatorname{Spaces})}(\mathcal{C})$ on $\mathcal{C}$ is called a {\rm sheaf of spaces} on $\mathcal{C}$ if for every object $C$ of $\mathcal{C}$ and for every monomorphism $U\hookrightarrow j(C)$ corresponding to a covering sieve on $C$ the induced map $\operatorname{Map}_{\operatorname{Preshv}_{(\operatorname{Spaces})}(\mathcal{C})}(j(C),F)\stackrel{\sim}{\to}\operatorname{Map}_{\operatorname{Preshv}_{(\operatorname{Spaces})}}(U,F)$ is a weak equivalence. 
The full subcategory $\operatorname{Shv}_{(\operatorname{Spaces})}(S)\subset\operatorname{Preshv}_{(\operatorname{Spaces})}(S)$ of sheaves of spaces on $\mathcal{C}$ is an $\infty$-topos (\cite{htt}, 6.2.2.7).  
\end{ex}
An $\infty$-category equipped with a Grothendieck topology is called an \textit{$\infty$-site}. 
An ordinary site can be seen as an $\infty$-site by taking the nerve. 

\begin{df}[Homotopy sheaves; \cite{htt}, 6.5.1] 
\label{homotopysheaves}
Let $\mathfrak{X}\subset\operatorname{Preshv}_{(\operatorname{Spaces})}(\mathcal{C})$ be an $\infty$-topos and $X$ a pointed object. 
For each non-negative integer $n\geq0$, the {\rm $n$-th homotopy sheaf} of $X$ is the sheaf of sets on $\mathcal{C}$ given by sheafifying the presheaf of sets on $\mathcal{C}$ that assigns to each object $C$ of $\mathcal{C}$ the $n$-th homotopy set $\pi_n(X(C))$ of the pointed space $X(C)$. 
\end{df}
We say a pointed object to be \textit{connected} if its $0$-th homotopy sheaf is trivial. 

Write $\Delta_{\operatorname{big}}$ for the category of non-empty finite linearly ordered sets. 
A \textit{simplicial object} in an $\infty$-topos $\mathfrak{X}$ is a functor $N(\Delta_{\operatorname{big}}^{\operatorname{op}})\to\mathfrak{X}$. 
The notions of group objects and their actions in an $\infty$-topos are formulated in terms simplicial objects, as follows.  
\begin{df}[group object; \cite{htt}, 6.1.2.7, 7.2.2.1]
\label{group}
A {\rm group object} of an $\infty$-topos $\mathfrak{X}$ is a simplicial object $G:N(\Delta_{\operatorname{big}}^{\operatorname{op}})\to\mathfrak{X}$ in $\mathfrak{X}$ such that $G([0])$ is a terminal object of $\mathfrak{X}$ and for every $n\geq0$ and for every partition $[n]=S\cup S^{\prime}$ with $S\cap S^{\prime}=\{s\}$, the maps $G([n])\to G(S)$ and $G([n])\to G(S^{\prime})$ exhibit $G([n])$ as a product of $G(S)$ and $G(S^{\prime})$.  
\end{df}
By a slight abuse of language we usually refer to the object $G([1])\in\operatorname{ob}\mathfrak{X}$ as a group object and call the simplicial object $G$ as the \textit{group structure} on $G([1])$. 

\begin{thm}[\cite{htt}, 7.2.2.11-(1)] 
\label{classifyingspace}
If $X$ is a connected pointed object of an $\infty$-topos $\mathfrak{X}$ then its loop space $\Omega X=\varprojlim(\ast\rightarrow X\leftarrow\ast)$ has a natural structure of a group object. 
This assignment $X\mapsto \Omega X$ arranges into an equivalence $$\Omega:\mathfrak{X}_{\ast,\operatorname{conn}}\leftrightarrows\operatorname{Grp}(\mathfrak{X}):B$$ between the $\infty$-categories $\mathfrak{X}_{\ast,\operatorname{conn}}$ of connected pointed objects of $\mathfrak{X}$ and $\operatorname{Grp}(\mathfrak{X})$ of group objects of $\mathfrak{X}$. 
The inverse functor $B$ takes a group object $G$ to the colimit $BG=\varinjlim G$ with the pointing given by $\ast=G([0])\to\varinjlim G$, where $G$ is seen as a diagram in $\mathfrak{X}$ indexed by $N(\Delta_{\operatorname{big}}^{\operatorname{op}})$. 
\end{thm}

\begin{df}[Action of a group object; \cite{nss}, Definition 3.1]
\label{action}
Let $G$ be a group object of an $\infty$-topos $\mathfrak{X}$. 
An {\rm action} of $G$ on an object $P\in\operatorname{ob}\mathfrak{X}$ is a map of simplicial objects $\rho\to G$ in $\mathfrak{X}$ such that $\rho([0])=P$ and for every $n\geq0$ and for every partition $[n]=S\cup S^{\prime}$ with $S\cap S^{\prime}=\{s\}$, the maps $\rho([n])\to\rho(S)$ and $\rho([n])\to G(S^{\prime})$ exhibit $\rho([n])$ as a product of $\rho(S)$ and $G(S^{\prime})$.  
\end{df}
Given an action $\rho\to G$ of $G$ on $P$, we get a square 
\begin{displaymath}
\begin{CD}
P@>>>\ast\\
@VVV@VVV\\
\varinjlim\rho@>>>BG 
\end{CD}
\end{displaymath}
by taking the colimits of the simplicial objects $\rho$ and $G$ seen as diagrams in $\mathfrak{X}$ indexed by $N(\Delta_{\operatorname{big}}^{\operatorname{op}})$. 
It can be shown that this square is a pullback square (\cite{nss}, Proposition 3.15). 
Conversely, given a pullback square 
\begin{displaymath}
\begin{CD}
P@>>>\ast\\
@VVV@VVV\\
X@>>>BG 
\end{CD}
\end{displaymath}
we can form a map of simplicial objects $\check{C}(P\to X)\to\check{C}(\ast\to BG)=G$ by taking the Cech nerves $\check{C}$ (see \cite{htt}, 6.1.2) of $P\to X$ and $\ast\to BG$. 
The constructions given above are mutually inverses to each other, due to the Giraud axiom saying that in an $\infty$-topos every groupoid object is effective; see \cite{nss}, section 3, for a details. 
Therefore, in an $\infty$-topos, giving an action of a group object $G$ on an object $P$ is equivalent to giving a fiber sequence $P\to X\to BG$, i.e. to describing $P$ as a pullback $P=\varprojlim(X\rightarrow BG\leftarrow\ast)$.   
\begin{df}[Torsor; \cite{nss}, Definition 3.4] 
\label{torsor}
Let $G$ be a group object in an $\infty$-topos $\mathfrak{X}$ and $X$ an object.  
A {\rm $G$-torsor over $X$} is a $G$-action $\rho\to G$ together with a map $\rho([0])\to X$ such that the induced map to $X$ from the colimit $\varinjlim\rho$, taken over the simplicial index category $N(\Delta^{\operatorname{op}}_{\operatorname{big}})$, is an equivalence.  
\end{df}
It is notable that this simple definition automatically implies, in the setting of an $\infty$-topos, the usual conditions for torsors, such as the principality condition and the local triviality. 
See \cite{nss}, Propositions 3.7 and 3.13. 

Given a $G$-torsor $\rho\to G$ over $X$, we get, by taking the colimits, a pullback square 
\begin{displaymath}
\begin{CD}
P@>>>\ast\\
@VVV@VVV\\
X@>>>BG,  
\end{CD}
\end{displaymath}
where $P=\rho([0])$, and in particular a map $X\to BG$. 
The above discussion on group actions shows that one can conversely construct a $G$-torsor $\check{C}(X\times_{BG}\ast\to X)\to G$ out of a given map $X\to BG$, and these constructions are mutually inverses. 
Hence, 
\begin{thm}[\cite{nss}, Theorem 3.19] 
\label{3.19}
Let $\mathfrak{X}$ be an $\infty$-topos and $G$ a group object. 
The $\infty$-category (which can be shown to be an $\infty$-groupoid; \cite{nss}, Proposition 3.18) of $G$-torsors over a fixed object $X$ is equivalent to the $\infty$-groupoid $\operatorname{Map}_{\mathfrak{X}}(X,BG)$ of maps from $X$ to $BG$. 
\end{thm}
In this sense we call $BG$ the \textit{classifying space object} of the group object $G$, and say that a map $X\to BG$ \textit{classifies} the $G$-torsor $\check{C}(X\times_{BG}\ast\to X)\to G$ over $X$. 
This and Theorem \ref{classifyingspace} exhibit a feature of $\infty$-topos theory, which is particularly convenient for our purposes, that in an $\infty$-topos the classifying space for torsors is just given by the connected delooping of the group. 

\section{Proofs}
Let $R$ be a commutative noetherian ring of finite Krull dimension, and consider the objects $\mathcal{K}$ and $\mathcal{K}_{\operatorname{Tate}}$ (Definitions \ref{k} and \ref{ktate}) of the $\infty$-topos $\operatorname{Shv}_{(\operatorname{Spaces})}(N\operatorname{Spec}R_{\operatorname{Nis}})$ of sheaves of spaces on the small Nisnevich site of $\operatorname{Spec}R$. 
\begin{lemma}
\label{pointedconnected}
The object $\mathcal{K}_{\operatorname{Tate}}$ is a connected pointed object of the $\infty$-topos $\operatorname{Shv}_{(\operatorname{Spaces})}(N\operatorname{Spec}R_{\operatorname{Nis}})$. 
\end{lemma}
\begin{proof}
The pointedness is trivial, with the pointing $\operatorname{Spec}R\to\mathcal{K}_{\operatorname{Tate}}$ classified by the point $[0]\in\mathcal{K}_{\operatorname{Tate}}(\operatorname{Spec}R)=\Omega^{\infty}\mathbb{K}((\displaystyle\lim_{\longleftrightarrow}\mathcal{P}(R))^{\natural})$ given by the chosen $0$-object of the exact category $(\displaystyle\lim_{\longleftrightarrow}\mathcal{P}(-))^{\natural}$. 
The nontrivial part is the connectedness, which amounts to showing that the $0$-th homotopy sheaf $\pi_0\mathcal{K}_{\operatorname{Tate}}$ is a terminal object.  
The sheaf of sets $\pi_0\mathcal{K}_{\operatorname{Tate}}$ is by Definition \ref{homotopysheaves} the sheafification of the presheaf $R^{\prime}\mapsto\pi_0(\mathcal{K}_{\operatorname{Tate}}(R^{\prime}))=\pi_0(\Omega^{\infty}\mathbb{K}((\displaystyle\lim_{\longleftrightarrow}\mathcal{P}(R^{\prime}))^{\natural}))$, which is naturally isomorphic to the presheaf $R^{\prime}\mapsto\pi_0(\Omega^{\infty}\Sigma\mathbb{K}(\mathcal{P}(R^{\prime})))=K_{-1}(R^{\prime})$ by Theorem \ref{abstractdelooping}. 
Now, it is a theorem of Drinfeld (\cite{drinfeld}, Theorem 3.4) that the presheaf $K_{-1}$ vanishes Nisnevich locally. 
Therefore its sheafification vanishes and we get the desired triviality of the $0$-th homotopy sheaf $\pi_0\mathcal{K}_{\operatorname{Tate}}$. 
\end{proof}
\begin{lemma}
\label{loopspace}
The loop space $\Omega\mathcal{K}_{\operatorname{Tate}}$ of the pointed object $\mathcal{K}_{\operatorname{Tate}}$ is equivalent to $\mathcal{K}$. 
\end{lemma}
\begin{proof}
Recall that the objects $\mathcal{K}$ and $\mathcal{K}_{\operatorname{Tate}}$ of $\operatorname{Shv}_{(\operatorname{Spaces})}(N\operatorname{Spec}R_{\operatorname{Nis}})$ are the images by $\theta$ of the simplicial presheaves $\Omega^{\infty}\mathbb{K}(-)$ and $\Omega^{\infty}\mathbb{K}((\displaystyle\lim_{\longleftrightarrow}\mathcal{P}(-))^{\natural})$ (Definitions \ref{k} and \ref{ktate}). 
In the simplicial category $(\operatorname{Set}_{\Delta}^{\operatorname{Spec}R_{\operatorname{Nis}}^{\operatorname{op}}})^{\circ}$ we have that the object $\Omega^{\infty}\mathbb{K}(-)$ is equivalent to the homotopy limit $\displaystyle{\rm \mathop{holim}_{\longleftarrow}}(\ast\rightarrow\Omega^{\infty}\mathbb{K}((\displaystyle\lim_{\longleftrightarrow}\mathcal{P}(-))^{\natural})\leftarrow\ast)$ by Theorem \ref{abstractdelooping}. 
By Theorem 4.2.4.1 of \cite{htt} this translates into an equivalence in $N(\operatorname{Set}_{\Delta}^{\operatorname{Spec}R_{\operatorname{Nis}}^{\operatorname{op}}})^{\circ}=\operatorname{Preshv}_{(\operatorname{Spaces})}(N\operatorname{Spec}R_{\operatorname{Nis}})$ between the object $\mathcal{K}=\theta(\Omega^{\infty}\mathbb{K}(-))$ and the limit $\varprojlim(\ast\rightarrow\theta(\Omega^{\infty}\mathbb{K}((\displaystyle\lim_{\longleftrightarrow}\mathcal{P}(-))^{\natural})\leftarrow\ast)$, which is by definition the loop space of the pointed object $\theta(\Omega^{\infty}\mathbb{K}((\displaystyle\lim_{\longleftrightarrow}\mathcal{P}(-))^{\natural})=\mathcal{K}_{\operatorname{Tate}}$. 
\end{proof}
{\bf Proof of Proposition \ref{Kasagroup}, Theorem \ref{geometricdelooping}, and Corollary \ref{dm}. }  
Recall Theorem \ref{classifyingspace} saying that for an $\infty$-topos $\mathfrak{X}$ there is an equivalence $$\Omega:\mathfrak{X}_{\ast,\operatorname{conn}}\leftrightarrows\operatorname{Grp}(\mathfrak{X}):B$$ between the $\infty$-categories $\mathfrak{X}_{\ast,\operatorname{conn}}$ of connected pointed objects of $\mathfrak{X}$ and $\operatorname{Grp}(\mathfrak{X})$ of group objects of $\mathfrak{X}$. 
For $\mathfrak{X}=\operatorname{Shv}_{(\operatorname{Spaces})}(N\operatorname{Spec}R_{\operatorname{Nis}})$ we have by Lemma \ref{pointedconnected} that $\mathcal{K}_{\operatorname{Tate}}$ is in $\mathfrak{X}_{\ast,\operatorname{conn}}$, and by Lemma \ref{loopspace} that its loop space $\Omega\mathcal{K}_{\operatorname{Tate}}$ is equivalent to $\mathcal{K}$. 
This provides with $\mathcal{K}$ the desired group structure, and the proof of Proposition \ref{Kasagroup} is complete. 

Applying the inverse functor $B$ to the equivalence between $\Omega\mathcal{K}_{\operatorname{Tate}}$ and $\mathcal{K}$ we obtain the desired equivalence between $\mathcal{K}_{\operatorname{Tate}}\cong B\Omega\mathcal{K}_{\operatorname{Tate}}$ and $B\mathcal{K}$, where $B\mathcal{K}$ serves as the classifying space object for $\mathcal{K}$-torsors in view of Theorem \ref{3.19}, and the proof of Theorem \ref{geometricdelooping} is complete. 

Thus we see that $\mathcal{K}$-torsors over $\operatorname{Spec}R$ are classified by maps from $\operatorname{Spec}R$ to $B\mathcal{K}\cong\mathcal{K}_{\operatorname{Tate}}$, which correspond by Yoneda's lemma to points of $\mathcal{K}_{\operatorname{Tate}}(\operatorname{Spec}R)=\Omega^{\infty}\mathbb{K}((\displaystyle\lim_{\longleftrightarrow}\mathcal{P}(R))^{\natural})$, thereby proving the first assertion of Corollary \ref{dm}. 
To prove the second assertion, let $M$ be a Tate vector bundle over $\operatorname{Spec}R$. 
Then as an object of the exact category $(\displaystyle\lim_{\longleftrightarrow}\mathcal{P}(R))^{\natural}$ it defines a point $[M]$ of the $K$-theory space $\Omega^{\infty}\mathbb{K}((\displaystyle\lim_{\longleftrightarrow}\mathcal{P}(R))^{\natural})=\mathcal{K}_{\operatorname{Tate}}(\operatorname{Spec}R)$, whose corresponding map $\operatorname{Spec}R\to\mathcal{K}_{\operatorname{Tate}}\cong B\mathcal{K}$ we also denote by $[M]$ by a slight abuse of notation.   
The desired torsor $\mathfrak{D}_M$ is the $\mathcal{K}$-torsor classified by this map $[M]$. 
Namely, it is the pullback $\mathfrak{D}_M=\varprojlim(\operatorname{Spec}R\stackrel{[M]}{\rightarrow}B\mathcal{K}\leftarrow\ast)$: 
\begin{displaymath}
\begin{CD}
\mathfrak{D}_M@>>>\operatorname{Spec}R\\
@VVV@VV{\text{base-point}}V\\
\operatorname{Spec}R@>{[M]}>>B\mathcal{K}. 
\end{CD}
\end{displaymath}
{\bf Proof of Theorem \ref{autmaction}. }
Let $M\in\operatorname{ob}(\displaystyle\lim_{\longleftrightarrow}\mathcal{P}(R))^{\natural}$ be a Tate vector bundle over $\operatorname{Spec}R$. 
We consider the simplicial presheaf on $\operatorname{Spec}R_{\operatorname{Nis}}$ that assigns to $R^{\prime}$ the simplicial set $N\overline{\operatorname{Aut}_{(\displaystyle\lim_{\longleftrightarrow}\mathcal{P}(R^{\prime}))^{\natural}}M\otimes_RR^{\prime}}$, where $\overline{\operatorname{Aut}_{(\displaystyle\lim_{\longleftrightarrow}\mathcal{P}(R^{\prime}))^{\natural}}M\otimes_RR^{\prime}}$ is the groupoid with a single object and morphisms on the unique object given by the elements of the group $\operatorname{Aut}_{(\displaystyle\lim_{\longleftrightarrow}\mathcal{P}(R^{\prime}))^{\natural}}M\otimes_RR^{\prime}$, and composition defined by the multiplication of the group. 
By taking the $\theta$ of the fibrant replacement of it we get a presheaf of spaces on $\operatorname{Spec}R^{\operatorname{op}}_{\operatorname{Nis}}$, which is denoted by $N\overline{\operatorname{Aut}M}$, and whose sheafification is denoted by $a(N\overline{\operatorname{Aut}M})$. 
We use the following lemma. 
\begin{lemma}
The classifying space object for the group object $\operatorname{Aut}M$ is given by $a(N\overline{\operatorname{Aut}M})$.   
\end{lemma}
\begin{proof}
The proof goes similarly to the proof of Theorem \ref{geometricdelooping}, once we notice that $a(N\overline{\operatorname{Aut}M})$ is a connected pointed object with its loop space object equivalent to $\operatorname{Aut}M$. 
With the obvious pointing $\ast\to N\overline{\operatorname{Aut}M}$ we have that $N\overline{\operatorname{Aut}M}$ is a pointed object, and so is its sheafification $a(N\overline{\operatorname{Aut}M})$.  
Recall the general fact that for every ordinary group $G$, the Kan complex $N\overline{G}$ is the Eilenberg-MacLane space $K(G,1)$, where $\overline{G}$ denotes the groupoid with a single object and morphisms given by elements of $G$. 
The $0$-th homotopy sheaf $\pi_0a(N\overline{\operatorname{Aut}M})$ is given by sheafifying the presheaf $R^{\prime}\mapsto\pi_0(N\overline{\operatorname{Aut}M}(R^{\prime}))$, and this vanishes since the Eilenberg-MacLane space $N\overline{G}=K(G,1)$ is always connected. 
Since the sheafification functor commutes with finite limits, the loop space $\Omega(a(N\overline{\operatorname{Aut}M}))$ is the sheafification of the loop space $\Omega(N\overline{\operatorname{Aut}M})$, which can be computed as the homotopy limit in the simplicial category $(\operatorname{Set}_{\Delta}^{\operatorname{Spec}R_{\operatorname{Nis}}^{\operatorname{op}}})^{\circ}$ by Theorem 4.2.4.1 of \cite{htt}.  
This in turn can be computed object-wise on the simplicial presheaf $R^{\prime}\mapsto N\overline{\operatorname{Aut}_{(\displaystyle\lim_{\longleftrightarrow}\mathcal{P}(R^{\prime}))^{\natural}}M\otimes_RR^{\prime}}=K(\operatorname{Aut}_{(\displaystyle\lim_{\longleftrightarrow}\mathcal{P}(R^{\prime}))^{\natural}}M\otimes_RR^{\prime},1)$, 
and using the general fact that $\Omega K(G,1)=G$ we get the desired conclusion. 
\end{proof} 
Now, recall that for a group object $G$ of an $\infty$-topos $\mathfrak{X}$, giving a $G$-action on an object $P\in\operatorname{ob}\mathfrak{X}$ is equivalent to giving a fiber sequence $P\to X\to BG$, i.e. to describing $P$ as a pullback $P=\varprojlim(X\rightarrow BG\leftarrow\ast)$. 
(See Definition \ref{action} and following discussions in section 2.)  
  
Hence, constructing the desired the $\operatorname{Aut}M$-action on the $\mathcal{K}$-torsor $\mathfrak{D}_M$ amounts to describe $\mathfrak{D}_M$ as a pullback $\mathfrak{D}_M=\varprojlim(X\rightarrow B\operatorname{Aut}M\leftarrow\operatorname{Spec}R)$ for some $X$ and some map $X\to B\operatorname{Aut}M$. 
It turns out that it suffices to have a map $[[M]]:B\operatorname{Aut}M\to B\mathcal{K}$ whose precomposition with the base-point map $\operatorname{Spec}R\to B\operatorname{Aut}M$ is equivalent to the map $[M]:\operatorname{Spec}R\to B\mathcal{K}$ classifying the $\mathcal{K}$-torsor $\mathfrak{D}_M$. 
Indeed, the successive pullback $\varprojlim(X\rightarrow B\operatorname{Aut}M\stackrel{(\text{base-point})}{\leftarrow}\operatorname{Spec}R)$, where $X=\varprojlim(B\operatorname{Aut}M\stackrel{[[M]]}{\rightarrow}B\mathcal{K}\stackrel{(\text{base-point})}{\leftarrow}\operatorname{Spec}R)$, is given by $\mathfrak{D}_M$ if $[[M]]\circ(\text{base-point})=[M]$:  
\begin{displaymath}
\begin{CD}
\mathfrak{D}_M@>>> X@>>>\operatorname{Spec}R\\
@VVV@VVV@VV{\text{base-point}}V\\
\operatorname{Spec}R@>{\text{base-point}}>>B\operatorname{Aut}M@>{[[M]]}>>B\mathcal{K}.\\
\end{CD}
\end{displaymath}
We thus get a fiber sequence $\mathfrak{D}_M\to\operatorname{Spec}R\to B\operatorname{Aut}M$, i.e. the description $\mathfrak{D}_M=\varprojlim(X\rightarrow B\operatorname{Aut}M\leftarrow\operatorname{Spec}R)$, as desired. 

To find such a map $[[M]]$, we notice that, in general, for any idempotent complete exact category $\mathcal{A}$ and an object $a$ of $\mathcal{A}$ the space $N\overline{\operatorname{Aut}_{\mathcal{A}}a}$ admits a natural, canonical map to the space $\Omega\abs{iS_{\bullet}(\mathcal{A})}=\Omega^{\infty}\mathbb{K}(\mathcal{A})$, where $S_{\bullet}$ denotes Waldhausen's $S_{\bullet}$-construction (\cite{waldhausen}, 1.3), $i(-)$ the subcategory of isomorphisms, and $\abs{-}$ the geometric realization.   
This is the composition of the map $N\overline{\operatorname{Aut}_{\mathcal{A}}a}\to Ni\mathcal{A}$ (recall that we write $i\mathcal{A}$ for the subcategory of isomorphisms) with the first structure map $Ni\mathcal{A}\to\Omega \abs{iS_{\bullet}(\mathcal{A})}$ of Waldhausen's connective algebraic $K$-theory spectrum (\cite{waldhausen}, 1.3). 
Applying this construction to $M\otimes_RR^{\prime}\in\operatorname{ob}(\displaystyle\lim_{\longleftrightarrow}\mathcal{P}(R^{\prime}))^{\natural}$ for \'{e}tale $R$-algebras $R^{\prime}$, we get a map of simplicial presheaves $N\overline{\operatorname{Aut}_{(\displaystyle\lim_{\longleftrightarrow}\mathcal{P}(-))^{\natural}}M\otimes_R(-)}\to\Omega^{\infty}\mathbb{K}((\displaystyle\lim_{\longleftrightarrow}\mathcal{P}(-))^{\natural})$.  
Via the equivalence $\theta$ this defines a map $N\overline{\operatorname{Aut}M}\to\mathcal{K}_{\operatorname{Tate}}$ in $\operatorname{Preshv}_{(\operatorname{Spaces})}(N\operatorname{Spec}R_{\operatorname{Nis}})$, which in turn induces a map $[[M]]:B\operatorname{Aut}M\cong a(N\overline{\operatorname{Aut}M})\to\mathcal{K}_{\operatorname{Tate}}\cong B\mathcal{K}$. 
Note that the precomposition of the map of simplicial presheaves $N\overline{\operatorname{Aut}_{(\displaystyle\lim_{\longleftrightarrow}\mathcal{P}(-))^{\natural}}M\otimes_R(-)}\to\Omega^{\infty}\mathbb{K}((\displaystyle\lim_{\longleftrightarrow}\mathcal{P}(-))^{\natural})$ with the canonical pointing $\operatorname{Spec}R\to N\overline{\operatorname{Aut}_{(\displaystyle\lim_{\longleftrightarrow}\mathcal{P}(-))^{\natural}}M\otimes_R(-)}$ corresponds to the point $[M]\in\Omega\mathbb{K}((\displaystyle\lim_{\longleftrightarrow}\mathcal{P}(R))^{\natural})$, so that the map $[[M]]$ satisfies the desired property $[[M]]\circ(\text{base-point})=[M]$. 
The proof of Theorem \ref{autmaction} is complete.

\begin{flushleft} 
Sho Saito\\
Graduate School of Mathematics\\
Nagoya University\\
Nagoya, Japan\\
\texttt{sho.saito@math.nagoya-u.ac.jp}
\end{flushleft}
\end{document}